\documentclass[leqno]{amsart}
\usepackage{amsmath,amsfonts,amsthm,amssymb,indentfirst,epic,url}

\newcommand{\Z}{\mathbb{Z}}
\newcommand{\Q}{\mathbb{Q}}
\newcommand{\lm}{\varprojlim}

\newtheorem{theorem}{Theorem}
\newtheorem{lemma}[theorem]{Lemma}
\newtheorem{corollary}[theorem]{Corollary}

\newtheorem{question}[theorem]{Question}
\newtheorem{example}[theorem]{Example}

\begin{document}

\begin{center}
\texttt{Comments, corrections, and related references welcomed,
as always!}\\[.5em]
{\TeX}ed \today
\vspace{2em}
\end{center}

\title%
{Every module is an inverse limit of injectives}
\thanks{This preprint is accessible online at
\url{http://math.berkeley.edu/~gbergman/papers/}
and at \url{http://arXiv.org/abs/arXiv:1104.3173}\,.
The former version is likely to be updated more frequently than
the latter.
}


\author{George M. Bergman}
\address{University of California\\
Berkeley, CA 94720-3840, USA}
\email{gbergman@math.berkeley.edu}

\begin{abstract}
It is shown that any left module $A$ over a ring
$R$ can be written as the intersection of a downward
directed system of injective submodules of an injective module;
equivalently, as an inverse limit of one-to-one homomorphisms
of injectives.
If $R$ is left Noetherian, $A$ can also
be written as the inverse limit of
a system of surjective homomorphisms of injectives.

Some questions are raised.
\end{abstract}
\maketitle

The flat modules over
a ring are precisely the direct limits of projective modules
\cite{DL} \cite{VEG} \cite[Theorem~2.4.34]{TYL}.
Which modules are, dually, inverse limits of injectives?

I sketched the answer in~\cite{250B_FE},
but in view of the limited distribution of that item, 
it seems worthwhile to make the result more widely available.
The construction from~\cite{250B_FE} is Theorem~\ref{T.submods} below;
the connecting maps there are inclusions.
In Theorem~\ref{T.onto}, we shall see that the connecting
maps can, alternatively, be
taken to be onto, if $R$ is Noetherian on the appropriate side.

In \S\ref{S.questions+} we ask some questions, in \S\ref{S.onto_ctbl}
we take some steps toward answering one
of them, and in \S\ref{S.gen} we note what the
proofs of our results tell us when applied to not necessarily injective
modules.

Throughout, ``ring'' means associative ring with unit,
and modules are unital.

I am indebted to Pace Nielsen for pointing out the need to
assume $\kappa$ regular in Lemma~\ref{L.prod^kappa},
and to the referee for some useful suggestions.

\section{Main results}\label{S.main}
We will need the following generalization of the familiar
observation (\cite[Proposition~I.3.1]{C+E},
\cite[Proposition~IV.3.7]{Hungerford})
that a direct product of injective modules is injective.
(It is a generalization because on taking $\kappa>\mathrm{card}(I),$
it yields that result.)

\begin{lemma}\label{L.prod^kappa}
Let $R$ be a ring, $\kappa$ an infinite regular cardinal such that
every left ideal of $R$ can be generated by $<\kappa$ elements,
and $(M_i)_{i\in I}$ a family of left $\!R\!$-modules.
Let
\begin{equation}\begin{minipage}[c]{35pc}\label{d.prod^kappa}
$\prod^\kappa_I\,M_i\ =\ \{x\in\prod_I M_i\,\mid\,x$ has
support of cardinality $<\kappa$ in $I\,\}.$
\end{minipage}\end{equation}
Then if all $M_i$ are injective, so is $\prod^\kappa_I M_i.$
\end{lemma}

\begin{proof}
To show $\prod^\kappa_I M_i$ injective, it suffices
by \cite[Theorem~I.3.2]{C+E} \cite[Lemma~IV.3.8]{Hungerford}
to show that for every left ideal $L$ of $R,$ every
module homomorphism $h:L\to\prod^\kappa_I M_i$
can be extended to all of $R.$
By choice of $\kappa,$ $L$ has a generating set $X$ of cardinality
$<\kappa,$ and by definition of $\prod^\kappa_I M_i,$
the image under $h$ of each member of $X$ has support of cardinality
$<\kappa $ in $I.$
Hence by regularity of $\kappa,$ the union $I_0\subseteq I$
of these supports has cardinality $<\kappa.$
Clearly $h(L)$ has support in $I_0,$ hence
$h$ can be looked at as a homomorphism $L\to\prod_{I_0} M_i.$
As each $M_i$ is injective, we can
now lift $h$ componentwise to a homomorphism
$R\to \prod_{I_0} M_i\subseteq\prod^\kappa_I M_i,$ as desired.
\end{proof}

\begin{theorem}\label{T.submods}
Let $R$ be a ring.
Then every left $\!R\!$-module $A$ can be written
as the intersection of a downward directed system of injective
submodules of an injective module; in other words, as the inverse limit
of a system of injective modules and one-to-one homomorphisms.
\end{theorem}

\begin{proof}
Given $A,$ choose an exact sequence of modules
\begin{equation}\begin{minipage}[c]{35pc}\label{d.seq}
$0\ \to\ A\ \to\ M\ \to\ N$
\end{minipage}\end{equation}
with $M$ and $N$ injective, as we may by \cite[Theorem~I.3.3]{C+E},
and call the second map $f:M\to N.$
Taking a cardinal $\kappa$ as in the preceding lemma
(for example, any infinite regular cardinal $>|R|),$ and a set $I$
of cardinality $\geq\kappa,$ one element of which we will
denote $0,$ we define $\!R\!$-modules $M_i$ $(i\in I)$
by letting $M_0=M,$ and $M_i=N$ for $i\neq 0.$

Now let $P=\prod^\kappa_I M_i,$ and for each finite subset
$D\subseteq I-\{0\},$ let $P_D\subseteq P$ be the submodule
of elements $(x_i)_{i\in I}$ such that for
all $i\in D,$ $x_i=f(x_0).$
Clearly, each element of $P_D$ is determined
by its components at the indices in $I-D,$ from which we see that
$P_D\cong\prod^\kappa_{I-D} M_i;$
so by Lemma~\ref{L.prod^kappa}, $P_D$ is injective.
The family of submodules $P_D$ is downward directed, since
$P_{D_1}\cap P_{D_2}=P_{D_1\cup D_2}.$

Now $\bigcap_D P_D\subseteq P$
consists of the elements $x\in P$ such that for
all $i\in I-\{0\},$ $x_i=f(x_0).$
Each such $x$ is determined by its coordinate $x_0\in M;$
but to lie in $P,$ such an element
must have support of cardinality $<\kappa,$
which only happens if $x_0\in\ker f.$
Thus, $\bigcap_D P_D\cong\ker f=A.$
\end{proof}

Note that in the construction of the above
proof, if $R$ is left Noetherian
then $\kappa$ can be taken to be $\aleph_0,$ and $I$ countable;
so the intersection is over the finite subsets of a countable set,
giving a countably indexed inverse system.
In that situation, $\prod^\kappa_I M_i$ is simply
$\bigoplus_I M_i,$ and Lemma~\ref{L.prod^kappa} then says that
the class of injective $\!R\!$-modules is closed under direct sums
(a known result, \cite[Proposition~2.1]{EM}.
In fact, that condition is necessary
and sufficient for $R$ to be left Noetherian
\cite[Theorem~1]{ZP} 
\cite[Theorem~1.1]{HB} 
\cite[Theorem~20.1]{CF},
a result variously called the Matlis-Papp Theorem, the
Cartan-Eilenberg-Bass Theorem, and by other combinations of these
names.)
We shall use this closure under
direct sums in the proof of our next theorem, along with
the following fact.
\begin{equation}\begin{minipage}[c]{35pc}\label{d.LH}
There exists an inverse system, indexed by the first uncountable
ordinal $\omega_1,$ of nonempty sets $S_\alpha$ and
surjective maps $f_{\alpha\beta}: S_\beta\to S_\alpha$
$(\alpha\leq\beta\in\omega_1),$ which has empty inverse limit
\cite{LH} \cite[\S2]{H+S} \cite{emptylim}.
\end{minipage}\end{equation}

Again, we begin with a general lemma.

\begin{lemma}[{after \cite[\S3]{H+S}; cf.\ \cite[Corollary~8]{emptylim}}]\label{L.zerolim}
Suppose $(S_\alpha,\,f_{\alpha\beta})_{\alpha\leq\beta\in\omega_1}$
is an inverse system of sets with the properties
stated in~\textup{(\ref{d.LH})}, and $N$ a
left module over a ring $R.$
To each $\alpha\in\omega_1,$ let us associate the direct sum
$\bigoplus_{S_\alpha}{\kern-.1em}N$ of an $\!S_\alpha\!$-tuple
of copies of $N;$ and for $\alpha\leq\beta,$ let $\varphi_{\alpha\beta}:
\bigoplus_{S_\beta}{\kern-.1em}N\to\bigoplus_{S_\alpha}{\kern-.1em}N$
be the map sending $(x_j)_{j\in S_\beta}$ to the element
$(y_i)_{i\in S_\alpha}$ with components
$y_i=\sum_{f_{\alpha\beta}(j)\,=\,i}\,x_j.$

Then each $\varphi_{\alpha\beta}$ is surjective, but
the inverse limit of the above system is the zero module.
\end{lemma}

\begin{proof}[Sketch of proof]
We imitate the argument of \cite{H+S}
(where $R$ was a field and $N$ was $R).$
Surjectivity of the $\varphi_{\alpha\beta}$ is clear.
Now suppose $x$ belongs to the inverse limit, and let us
write its components $x^{(\alpha)}\in\bigoplus_{S_\alpha}{\kern-.1em}N$
$(\alpha\in\omega_1).$
For each $\alpha\in\omega_1,$ let $T_\alpha\subseteq S_\alpha$
be the (finite) support of $x^{(\alpha)}.$
We see that the cardinalities of the $T_\alpha$ are monotonically
nondecreasing in $\alpha;$ hence, since
$\omega_1$ has uncountable cofinality, the supremum of
those cardinalities must be finite.
(Indeed, for each $n$ such that some $|T_\alpha|$ equals $n,$
let us choose an $\alpha_n$ realizing this value.
Then the at most countably many indices $\alpha_n$
have a supremum, $\alpha_{\sup}\in\omega_1,$
and the finite value $|T_{\alpha_{\sup}}|$ will bound all $|T_\alpha|.)$

Calling this supremum $n,$ we see that
the set of $\alpha\in\omega_1$ such that $|T_\alpha|=n$ is
an up-set in $\omega_1,$
and that whenever $\alpha\leq \beta$ are both in this up-set,
the connecting map $f_{\alpha\beta}$ gives a bijection
$T_\beta\to T_\alpha.$
These $\!n\!$-element sets $T_\alpha$ thus lead to an $\!n\!$-tuple of
elements of $\lm\,S_{\alpha}.$
But by assumption, that limit set is empty.
Hence $n=0,$ so all $x^{(\alpha)}$ are~$0,$ so $x=0.$
\end{proof}

We can now prove

\begin{theorem}\label{T.onto}
Let $R$ be a left Noetherian ring.
Then every left $\!R\!$-module $A$ can be written as the inverse limit
of a system, indexed by $\omega_1,$ of surjective homomorphisms
of injective modules.
\end{theorem}

\begin{proof}
Again let $f:M\to N$ be
a homomorphism of injective left $\!R\!$-modules with kernel $A.$
Let us take the inverse system of direct sums of copies of $N$
described in Lemma~\ref{L.zerolim}, and append to each of
these direct sums a copy of $M,$ getting modules
\begin{equation}\begin{minipage}[c]{35pc}\label{d.M+sum_N}
$M\oplus\,\bigoplus_{S_\alpha}{\kern-.1em}N$ $(\alpha\in\omega_1),$
\end{minipage}\end{equation}
which we connect using maps that act on $M$ as the identity,
and act on the direct sums of copies of $N$ by the connecting
morphisms of Lemma~\ref{L.zerolim}.
Assuming for notational convenience that none
of the $S_\alpha$ contains an element named $0,$
let us write the general element of~(\ref{d.M+sum_N})
as $(x_i)_{i\in\{0\}\cup S_\alpha},$
where $x_0\in M$ and the other components are in~$N.$

We now define, for each $\alpha\in\omega_1,$
\begin{equation}\begin{minipage}[c]{35pc}\label{d.P_i}
$P_\alpha\ =\ \{x=(x_i)_{i\in\{0\}\cup S_\alpha}\in
M\oplus\,\bigoplus_{S_\alpha}{\kern-.1em}N\ \mid
\ \sum_{i\in S_\alpha} x_i=f(x_0)\}.$
\end{minipage}\end{equation}

Note that for each $\alpha,$ if we choose any
$i_0\in S_\alpha,$
then we can specify an element $x\in P_\alpha$ by choosing
its components other than $x_{i_0}$ to comprise an arbitrary
member of $M\oplus\,\bigoplus_{S_\alpha-\{i_0\}}{\kern-.1em}N.$
The value of $x_{i_0}$ will then be determined by the relation
$\sum_{i\in S_\alpha} x_i=f(x_0).$
Thus,
$P_\alpha\cong M\oplus\,\bigoplus_{S_\alpha-\{i_0\}}{\kern-.1em}N,$
a direct sum of injectives, so
since $R$ is left Noetherian, each $P_\alpha$ is injective.
Clearly, the inverse system of surjective maps among the
modules~(\ref{d.M+sum_N}) induces an inverse system
of surjective maps among the submodules~(\ref{d.P_i}).

In a member of $\lm_{\omega_1}P_\alpha,$ the
$\bigoplus_{S_\alpha}\! N\!$-components, as $\alpha$
ranges over $\omega_1,$ will form a member of the inverse limit of
the system of Lemma~\ref{L.zerolim}; hence these components must all
be zero.
Thus, the corresponding $\!M\!$-components must belong to $\ker f=A.$
Since the connecting maps on these components
are the identity map of $M,$ the inverse limit is $A\subseteq M.$
\end{proof}


(Incidentally, Theorem~\ref{T.submods} or~\ref{T.onto} yields a
correct proof of \cite[Lemma~3]{LS}, the statement that $\Z$ is an
inverse limit of injective abelian groups.
The construction of \cite{LS} is similar to our proof of
Theorem~\ref{T.submods},
but since the groups $H_j$ used there are uniquely
$\!p\!$-divisible for all odd primes $p,$ their intersection
is $\!p\!$-divisible, and so is not $\Z.)$

For further examples of unexpectedly small inverse
limits, see~\cite{emptylim}, \cite{LH}, \cite{H+S}, \cite{WCW}.
Some questions about these are noted in \cite[\S\S4-5]{emptylim}.

\section{Questions}\label{S.questions+}
Theorem~\ref{T.onto} leaves open

\begin{question}\label{Q.nonNoe}
For non-left-Noetherian $R,$ which left $\!R\!$-modules are inverse
limits of systems of surjective maps of injective $\!R\!$-modules?
\textup{(}All?\textup{)}
Does the answer change if we restrict ourselves to systems indexed,
as in Theorem~\ref{T.onto}, by $\omega_1$?
\end{question}

We noted following Theorem~\ref{T.submods} that for $R$ Noetherian,
the construction used there involved a countable inverse system.
This suggests

\begin{question}\label{Q.1-1_ctbl}
For non-left-Noetherian $R,$ which left $\!R\!$-modules are
inverse limits of \emph{countable} systems of \emph{one-to-one} maps of
injective $\!R\!$-modules?
\textup{(}All?\textup{)}
\end{question}

On the other hand, the construction of Theorem~\ref{T.onto} used
uncountable inverse systems in all cases, and so leaves open

\begin{question}\label{Q.onto_ctbl}
For a \textup{(}left Noetherian or general\textup{)} ring $R,$
which left $\!R\!$-modules are inverse limits of \emph{countable}
systems of \emph{surjective} maps of injective left $\!R\!$-modules?
\end{question}

\section{Partial results on Question~\ref{Q.onto_ctbl}}\label{S.onto_ctbl}

The answer to Question~\ref{Q.onto_ctbl} cannot be either
``all modules'' or ``only the injectives'', even for $R=\Z,$ as will
be shown by Corollary~\ref{C.big-div} and Example~\ref{E.non-div},
respectively.

In describing inverse limits, we have indexed our inverse
systems so that the
connecting maps go from higher- to lower-indexed objects.
In direct limits, which appear beside
inverse limits in the following preparatory lemma, we shall take
the connecting maps to go from lower- to higher-indexed objects.
(Thus, in each kind of limit,
our index-sets are \emph{upward} directed.)

\begin{lemma}\label{L.proj-inj}
Let $R$ be a ring.
Let $M$ be the inverse limit of a
countable system of injective left $\!R\!$-modules $M_\alpha$
and surjective homomorphisms
$\varphi_{\alpha\beta}:M_\beta\to M_\alpha$
$(\alpha\leq\beta,\ \alpha,\beta\in I),$
and let $N$ be the \emph{direct limit} of a
countable system of \emph{projective} left $\!R\!$-modules $N_\gamma$
and \emph{one-to-one} homomorphisms
$\psi_{\delta\gamma}:N_\gamma\to N_\delta$
$(\gamma\leq\delta,\ \gamma,\delta\in J).$

Then any homomorphism
\begin{equation}\begin{minipage}[c]{35pc}\label{d.*g*a}
$f: N_\gamma\ \to\ M_\alpha,\ $ where $\gamma\in J,$ $\alpha\in I$
\end{minipage}\end{equation}
can be factored
\begin{equation}\begin{minipage}[c]{35pc}\label{d.viaMN}
$N_\gamma\ \to\ N\to\ M\to\ M_\alpha,$
\end{minipage}\end{equation}
where the first and last maps are the canonical ones
associated with the given direct and inverse limits \textup{(}and
the indices $\gamma$ and $\alpha),$ while the middle map
is an arbitrary module homomorphism.
\end{lemma}

\begin{proof}
Let us be given a homomorphism~(\ref{d.*g*a}).

Recall that every countable directed partially ordered set (or more
generally, any directed partially ordered set of countable cofinality)
has a cofinal chain isomorphic to $\omega,$ and that a direct
or inverse limit over the original set is isomorphic to the
corresponding construction over any such chain.
In our present situation, we can clearly take such a chain
in $I$ which begins with the index $\alpha$ of~(\ref{d.*g*a}),
and such a chain in $J$ beginning with the index $\gamma.$
Hence, replacing the two given systems with the systems determined
by these chains, we may assume that our direct and inverse
system are both indexed by $\omega,$ and name the map
we wish to extend $f_0: N_0\to M_0$ (see~(\ref{d.diagram}) below).

Using the projectivity of $N_0$ and the surjectivity
of $\varphi_{01}: M_1\to M_0,$ we can now factor
$f_0$ as $\varphi_{01}\,g_0$ for some homomorphism $g_0:N_0\to M_1,$
and then, similarly using the injectivity of $M_1$ and one-one-ness
of $\psi_{10}: N_0\to N_1,$ factor $g_0$ as
$f_1\,\psi_{10}$ for some $f_1:N_1\to M_1.$
Thus, we get $f_0=\varphi_{01}\,f_1\,\psi_{10}.$

We now iterate this process, getting $f_2:N_2\to M_2,$ etc.,
where each composite $N_{i-1}\to N_i\to M_i\to M_{i-1}$ is the
preceding map $f_{i-1}:$
\begin{equation}\begin{minipage}[c]{35pc}\label{d.diagram}
\begin{picture}(150,60)
\put(5,12){$\cdots$}
\multiput(25,15)(40,0){2}{\vector(1,0){18}}
\put(5,42){$\cdots$}
\multiput(41,45)(40,0){2}{\vector(-1,0){18}}
\put(52,38){\dottedline{2}(0,0)(0,-14)\put(0,-18){\vector(0,-1){0}}}
\put(45,42){$N_i$}
\put(45,10){$M_i$}
\put(53,28){$f_i$}
\put(80,0){
\put(5,12){$\cdots$}
\multiput(25,15)(40,0){3}{\vector(1,0){18}}
\put(5,42){$\cdots$}
\multiput(41,45)(40,0){3}{\vector(-1,0){18}}
\put(45,10){
\put(00,32){$N_2$}
\put(40,32){$N_1$}\put(20,42){$\psi_{21}$}
\put(80,32){$N_0$}\put(60,42){$\psi_{10}$}
\put(00,0){$M_2$}
\put(40,0){$M_1$}\put(20,-5){$\varphi_{12}$}
\put(80,0){$M_0\,.$}\put(60,-5){$\varphi_{01}$}
\put(08,18){
\put(00,0){$f_2$}
\put(40,0){$f_1$}
\put(80,0){$f_0$}
}
}
\put(52,38){
\put(00,0){\dottedline{2}(0,0)(0,-14)\put(0,-18){\vector(0,-1){0}}}
\put(40,0){\dottedline{2}(0,0)(0,-14)\put(0,-18){\vector(0,-1){0}}}
\put(80,0){\vector(0,-1){18}
}
}
}
\end{picture}
\end{minipage}\end{equation}
In particular, each composite
$N_0\to N_i\to M_i\to M_0$ is our original map $f_0.$
Using the universal properties of direct and inverse limits,
we see that these maps induce a map $N\to M$ such that the composite
$N_0\to N\to M\to M_0$ is $f_0,$ as required.
\end{proof}

Now suppose that $R$ is a commutative principal ideal domain.
Then it is easy to verify that an $\!R\!$-module $M$ is
injective if and only if it is \emph{divisible}, i.e., if
and only if it is a homomorphic image, as an $\!R\!$-module,
of some $\!K\!$-module, where $K$ is the field of fractions of $R.$
If, further, $R\neq K,$ and $R$ has at most countably many primes,
say $p_1,\,p_2,\,\dots$ (where we allow repetitions in
this list, in case $R$ has only finitely many), then $K$ is,
as an $\!R\!$-module, the direct limit of
a chain of inclusions of free $\!R\!$-modules of rank~$1$
\begin{equation}\begin{minipage}[c]{35pc}\label{d.big-div}
$R\ \subseteq\ p_1^{-1} R\ \subseteq
\ p_1^{-2}\,p_2^{-2} R\ \subseteq\ \cdots\ \subseteq
\ p_1^{-i}\,p_2^{-i}{\dots}\ p_i^{-i} R\ \subseteq \cdots\,.$
\end{minipage}\end{equation}
Hence we can apply Lemma~\ref{L.proj-inj} with $K$ as $N,$ calling
the modules of~(\ref{d.big-div}) $N_0\subseteq N_1\subseteq\cdots\,;$
but still letting $(M_\alpha)_{\alpha\in I}$ be an arbitary countable
inverse system of injectives.
For any $\alpha\in I,$ every $x\in M_\alpha$ is, of course,
the image of the generator $1\in R=N_0$
under some homomorphism $f: N_0\to M_\alpha.$
Hence Lemma~\ref{L.proj-inj} tells us that $x$ lies in
the image of a homomorphism $K=N\to M\to M_\alpha;$
so the span in $M$ of the images of all homomorphisms
$K\to M$ maps surjectively to each $M_\alpha.$
For brevity and concreteness, we state this result below for $R=\Z.$

\begin{corollary}\label{C.big-div}
Let $M$ be the inverse limit of a countable system
of injective $\!\Z\!$-modules $M_\alpha$ and surjective homomorphisms
$\varphi_{\alpha\beta}:M_\beta\to M_\alpha.$
Let $M_\mathrm{div}$ be the largest divisible
\textup{(}equivalently, injective\textup{)} submodule
of $M,$ namely, the sum of the images of all $\Z\!$-module
homomorphisms $\Q\to M.$
Then $M_\mathrm{div}$ projects surjectively to each $M_\alpha;$ i.e.,
the composite maps $M_\mathrm{div}\hookrightarrow M\to M_\alpha$
are surjective.\qed
\end{corollary}

This shows that if $M$ is nontrivial, it must have a sizable
injective submodule.
(In particular, $M$ cannot be a nonzero finitely
generated $\!\Z\!$-module.)
However, the following example shows
that that submodule need not be all of $M.$

\begin{example}\label{E.non-div}
A countable inverse system $\dots\to M_2\to M_1\to M_0$
of injective $\!\Z\!$-modules and
surjective homomorphisms, whose inverse limit $M$ is not injective.
\end{example}

\begin{proof}[Construction and proof]
For each $n\geq 0,$ let
\begin{equation}\begin{minipage}[c]{35pc}\label{d.Q+Q/Z}
$M_n\ =\ \Q\ \oplus\ \dots\ \oplus\ \Q\ \oplus\ (\Q/\Z)\ \oplus\ \dots
\ \oplus\ (\Q/\Z)\ \oplus\ \dots\,,$
\end{minipage}\end{equation}
where the summands $\Q$ are indexed by $i=0,\dots,n-1,$ and the
$\Q/\Z$ by the $i\geq n.$
Define connecting maps $\varphi_{mn}: M_n\to M_m$ $(m\leq n)$ to act
componentwise; namely, as the identity map of $\Q,$
respectively, of $\Q/\Z,$
on the components with indices $i<m$ or $i\geq n,$ and as
the reduction map $\Q\to\Q/\Z$ on the $n-m$ intermediate components.

It is not hard to verify that the inverse limit $M$
of these modules can be identified with the submodule of
$\Q^\omega$ consisting of those elements all but finitely many of whose
components lie in $\Z.$
(Given $x\in M,$ its image in $M_0$ will have all but finitely
many components $0\in\Q/\Z,$ and these zero components will
correspond to the components of $x$ which lie in $\Z.)$

If we take an element $x\in M$ and a positive integer $k$ such that
the entries of $x$ in $\Z$ are \emph{not}
almost all divisible by $k,$ then $x$ is not divisible by $k$ in $M.$
Hence $M$ is not a divisible group, i.e., is not injective.
\end{proof}

Returning to Corollary~\ref{C.big-div}, we remark that its
method of proof, applied
to a countable inverse limit $M$ of injective modules and
surjective homomorphisms over any integral domain $R,$
shows that $M$ contains many ``highly divisible'' elements.
For most $R,$ this shows that
not all $\!R\!$-modules can occur as such inverse limits.

\section{Not necessarily injective modules}\label{S.gen}
None of the constructions we have used to get an inverse
system of modules from an exact sequence $0\to A\to M\to N$
are limited to the case where $M$ and $N$ are injective.
Let us record what they give us in general.

\begin{corollary}[to proofs of Theorems~\ref{T.submods}
and~\ref{T.onto}, and~Example~\ref{E.non-div}]\label{C.gnl}
Let $R$ be a ring, $\mathbf{M}$ a class of left $\!R\!$-modules,
$\kappa$ an infinite regular cardinal such that $\mathbf{M}$ is closed
under $\!\kappa\!$-restricted direct products $\prod_I^\kappa M_\alpha,$
and $0\to A\to M\to N$ any exact sequence of left $\!R\!$-modules
with $M,\,N\in\mathbf{M}.$
Then\vspace{0.2em}

\textup{(a)} $A$ can be written as the inverse limit of
a system of modules in $\mathbf{M}$ and one-to-one
homomorphisms.\vspace{0.2em}

\textup{(b)} If $\kappa=\aleph_0$ \textup{(}so that the
hypothesis on $\mathbf{M}$ is that it is closed under direct
\emph{sums}\textup{)}, then $A$ can be written as the inverse limit of
an $\!\omega_1\!$-indexed system of modules in $\mathbf{M}$
and \emph{surjective} homomorphisms.\vspace{0.2em}

\textup{(c)} If, again, $\kappa=\aleph_0,$ then the submodule
of $M^\omega$ consisting of those
elements with all but finitely many components in $A$ can be
written as the inverse limit of a \emph{countable}
system of modules in $\mathbf{M}$ and surjective homomorphisms.\qed
\end{corollary}

So, for instance, by~(b), for any ring $R,$ any $\!R\!$-module
which can be written as the kernel of a homomorphism of projective
modules can also be written as the inverse limit of a system
of projective modules and surjective homomorphisms.



\end{document}